\documentclass[11pt]{amsproc}
 \usepackage[margin=1in]{geometry}
\usepackage{setspace,fullpage}
\usepackage{graphicx}
\usepackage[nice]{nicefrac}
\usepackage{amssymb}
\usepackage{amsmath,amsthm,amscd,amssymb, mathrsfs}

\usepackage{latexsym}

\numberwithin{equation}{section}

\theoremstyle{plain}
\newtheorem{theorem}{Theorem}[section]
\newtheorem{lemma}[theorem]{Lemma}
\newtheorem{corollary}[theorem]{Corollary}

\theoremstyle{definition}
\newtheorem{definition}[theorem]{Definition}

\theoremstyle{remark}
\newtheorem{remark}[theorem]{Remark}

\newtheorem{case[theorem]}{Case}

\title[\parbox{14cm}{\centering{ Averaging operators over nondegenerate quadratic surfaces in finite fields \hspace{1in}}} \quad]{ Averaging operators over nondegenerate quadratic surfaces in finite fields }
\author{ Doowon Koh }

\address{Department of Mathematics\\
Chungbuk National University \\
Cheongju city, Chungbuk-Do 361-763 Korea}
\email{koh131@chungbuk.ac.kr}

\thanks{
This research was supported by Basic Science Research Program through the National Research Foundation of Korea(NRF) funded by the Ministry of Education, Science and Technology(2012010487)
}


\begin{document}

\begin{abstract}  We study mapping properties of the averaging operator related to the variety $ V=\{x\in \mathbb F_q^d: Q(x)=0\},$ where $Q(x)$ is a nondegenerate quadratic polynomial over a finite field 
$\mathbb F_q$ with $q$ elements. This paper is  devoted to eliminating the logarithmic bound appearing in  the paper \cite{KS11}. As a consequence, we settle down the averaging problems over the quadratic surfaces $V$ in the case when  the dimensions $d\geq 4$ are even and $V$ contains a $d/2$-dimensional subspace.
\end{abstract} 
\maketitle

\section{Introduction}

Let $V \subset \mathbb R^d$ be a smooth hypersurface and $d\sigma$ a smooth, compactly supported surface measure on $V.$ An averaging operator $A$ over $V$ is given by 
$$Af(x)=f\ast d\sigma(x)=\int_V f(x-y) d\sigma(y)$$
where $f$ is a complex valued function on $\mathbb R^d.$
In this Euclidean setting, the averaging problem is to determine  the optimal range of exponents $ 1\leq p, r\leq \infty$ such that  
\begin{equation}\label{AI} \|f\ast d\sigma \|_{L^r(\mathbb R^d)}\leq C_{p,r,d} \|f\|_{L^p(\mathbb R^d)}, ~~f\in \mathcal{S}(\mathbb R^d)\end{equation}
where $\mathcal{S}(\mathbb R^d)$ denotes the space of Schwartz functions. When $V$ is the unit sphere $\mathbb S^{d-1},$  this problem is closely related to  regularity estimates of the solutions to the wave equation at time $t=1$,  and it was studied by R.S. Strichartz \cite{St70}. It is well known that $L^p-L^r$ averaging results can be obtained by  the decay estimates of the Fourier transform of the surface measure $d\sigma$ on $V$.
For instance, if $|\widehat{d\sigma}(\xi)|=\left| \int_V e^{-2\pi i x\cdot \xi} d\sigma(x)\right|\lesssim (1+|\xi|)^{-\alpha}$ for some $\alpha>0$, then the averaging inequality (\ref{AI}) holds whenever  
$$1\leq p\leq 2, \quad \frac{1}{p}-\frac{1}{2}\leq \frac{1}{2} \left(\frac{\alpha}{\alpha+1} \right), ~~\mbox{and}~~ 
 r=p^\prime,$$
where $p^\prime$ denotes the exponent conjugate to $p$ (see  \cite{{Li73},{St93}}). Thus, if $ |\widehat{d\sigma}(\xi)|\lesssim (1+|\xi|)^{-(d-1)/2}$ and $(1/p, 1/r)=( d/(d+1), 1/(d+1) ),$ then the averaging estimate (\ref{AI}) holds. Since $L^1-L^1$ and $L^\infty-L^\infty$ estimates are clearly possible,  we see from the interpolation theorem that  if  $ |\widehat{d\sigma}(\xi)|\lesssim (1+|\xi|)^{-(d-1)/2}, $  then $L^p-L^r$ estimates hold whenever $(1/p, 1/r)$ lies in the triangle $\Delta_d$ with vertices $(0,0), (1,1),$ and $(d/(d+1), 1/(d+1) ).$ Moreover, it is well known that  $L^p-L^r$ estimates are impossible if $(1/p,1/r)$ lies outside of the triangle $\Delta_d.$ Such analogous phenomena were also observed in the finite field setting (see, for example, \cite{{CSW08},{KS12}, {KS11}}).\\
 On the other hand,  if  the optimal Fourier decay estimate  of the surface measure $d\sigma$  is given by 
$$ |\widehat{d\sigma}(\xi)|\lesssim (1+|\xi|)^{-\alpha} \quad\mbox{for some}~~\alpha < (d-1)/2,$$
then it is in general hard to prove sharp averaging results and some technical arguments are required  to deal with the case (see \cite{IS96}).
In the finite field case,  this was also pointed out by the authors in \cite{CSW08}.\\

\noindent  As an analogue of Euclidean averaging problems,  Carbery, Stones, and Wright \cite{CSW08} initially introduced and studied the averaging problem in finite fields 
and the author with Shen has recently investigated the averaging problem over homogeneous varieties. In this introduction, we shortly review notation and definitions for averaging problems in the finite field setting and
readers are  referred to \cite{KS11} for further information and motivation on the averaging problem.
Let $\mathbb F_q^d$ be a $d$-dimensional vector space over a finite field $\mathbb F_q$ with $q$ elements.
We endow the space $\mathbb F_q^d$ with a normalized counting measure $``dx"$.
Let $V \subset \mathbb F_q^d$ be an algebraic variety.  Then a normalized surface measure $\sigma$ supported on $V$  can be defined by the relation
$$ \int f(x)~ d\sigma(x) = \frac{1}{|V|} \sum_{x\in V} f(x)$$
where $|V|$ denotes the cardinality of $V$ (see \cite{MT04}).
An averaging operator $A$ can be defined by 
$$ Af(x)=f\ast d\sigma(x)= \int f(x-y)~ d\sigma(y)= \frac{1}{|V|} \sum_{y\in V} f(x-y) $$
where both $f$ and $Af$ are functions on $(\mathbb F_q^d, dx).$
In this setting the averaging problem over $V$ is to determine  $1\leq p, r\leq \infty$ such that 
\begin{equation}\label{Apq} \|Af\|_{L^r(\mathbb F_q^d, dx) }\leq C \|f\|_{L^p(\mathbb F_q^d, dx)},\end{equation}
where the constant $C>0$ is independent of  functions $f$ and $q$, the cardinality of the underlying finite field $\mathbb F_q.$
\begin{definition} Let $1\leq p,r \leq \infty.$
We denote by  $ A(p\to r)\lesssim 1$ to  indicate that the averaging inequality (\ref{Apq}) holds.
\end{definition}

The main purpose of this paper is to obtain  the complete $L^p-L^r$ estimates of the averaging operators over varieties determined by  nondegenerate quadratic form over $\mathbb F_q$.
Let $Q(x) \in \mathbb F_q[x_1, \dots, x_d]$ be a nondegenerate quadratic form. 
Define a variety 
$$ S=\{x\in \mathbb F_q^d: Q(x)=0\}.$$
We shall name this kind of varieties as a nondegenerate quadratic surface in $\mathbb F_q^d.$
Since $Q(x)$ is a nondegenerate quadratic form,  it can be transformed into a diagonal form $a_1x_1^2+\dots+ a_dx_d^2$ with $a_j\neq 0$ by means of a linear substitution (see \cite{LN97}).
Therefore, we may assume that any nondegenerate quadratic surface can be written by the form
\begin{equation} \label{defv} S=\{x\in \mathbb F_q^d : a_1x_1^2 +\cdots+ a_d x_d^2=0\}\end{equation}
where  $a_j\in \mathbb F_q \setminus \{0\}, j=1,\dots,d.$   The necessary conditions for the averaging estimates over $S$ were given in \cite{CSW08, KS11}. In fact, 
 $A(p\to r)\lesssim 1 $ only if $(1/p, 1/r)$ lies in the convex hull of $(0,0), (0,1), (1,1),$ and $ (d/(d+1), 1/(d+1)).$ 
It is known from \cite{KS11} that  this necessary conditions for $A(p\to r)\lesssim 1$ are  sufficient conditions  if the dimension $d\geq 3$ is odd.
On the other hand, it was observed in \cite{KS11} that if  $d\geq 4$ is even and $S$ contains a subspace $H$ with $|H|=q^{d/2},$ then $A(p\to r)\lesssim 1$ only if $(1/p, 1/r)$ lies in the convex hull of 
\begin{equation}\label{imnec}(0,0), (0,1), (1,1), \left( \frac{d^2-2d+2}{d(d-1)}, ~\frac{1}{d-1}\right),~~\mbox{and} ~~\left( \frac{d-2}{d-1}, ~ \frac{d-2}{d(d-1)}\right). \end{equation}
In this paper we show that  (\ref{imnec})  is also the sufficient conditions for $A(p\to r)\lesssim 1$  in the specific case when the variety $S$ contains $d/2$-dimensional subspace with $d\geq 4$ even.
See Figure \ref{fig1}. 

\begin{figure}
\centering
\includegraphics[width=0.8\textwidth]{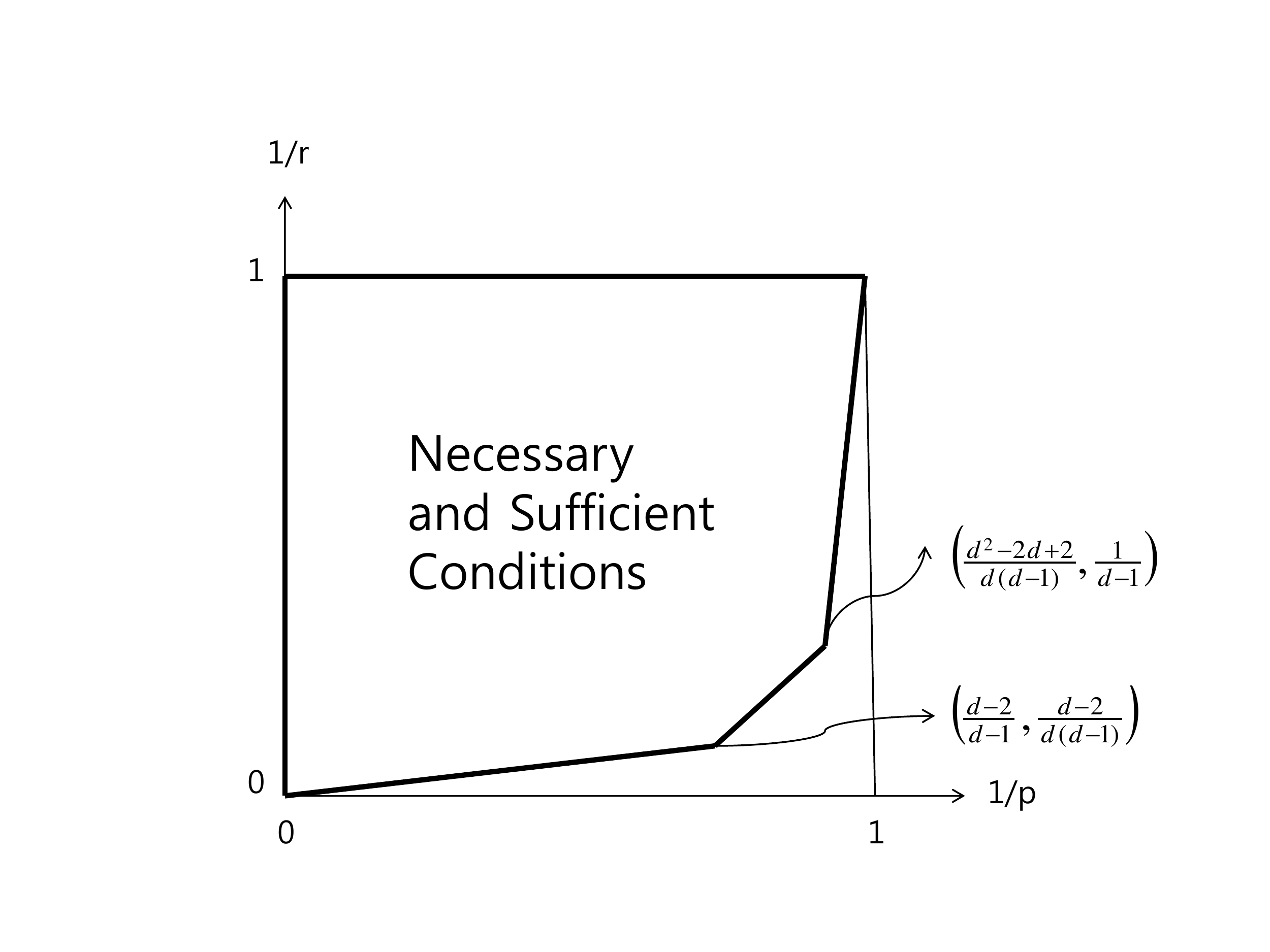}
\caption{Necessary and Sufficient Conditions for $A(p\to r)\lesssim 1$ in the case that $d\geq 4$ is even and $S$ contains a $d/2$-dimensional subspace}
\label{fig1}
\end{figure}
\subsection{Statement of main result}
\begin{theorem}\label{mainthm}
 Let $d\sigma$ be the normalized surface measure on the nondegenerate quadratic surface $S\subset \mathbb F_q^d,$ as defined  in (\ref{defv}).
Suppose that $d\geq 4$ is an even integer and $S$ contains a $d/2$-dimensional subspace. Then $A(p\to r)\lesssim 1$ if and only if 
$(1/p, 1/r)$ lies in the convex hull of 
$$(0,0), (0,1), (1,1), \left( \frac{d^2-2d+2}{d(d-1)}, ~\frac{1}{d-1}\right),~~\mbox{and} ~~\left( \frac{d-2}{d-1}, ~ \frac{d-2}{d(d-1)}\right).$$
\end{theorem}

\begin{remark} If the dimension $d\geq 4$ is even, then the diagonal entries $a_j$ can be properly chosen so that  $S$  contains a $d/2$-dimensional subspace $H.$  Such an example is essentially the following one (see Theorem 4.5.1 of \cite{Sc69}):
 $S=\{x\in \mathbb F_q^d: \sum_{k=1}^{d} (-1)^{k+1} x_k^2 =0\}$ and $H=\{(t_1, t_1, t_2, t_2, \dots, t_{d/2}, t_{d/2}) \in \mathbb F_q^d: t_1, t_2, \dots, t_{d/2} \in \mathbb F_q\}.$
\end{remark}

\begin{remark}\label{remgo} From the observation (\ref{imnec}), we only need to prove the `` if " part of Theorem \ref{mainthm}.  Since $dx$ is the normalized counting measure on $\mathbb F_q^d$, it follows from Young's inequality that  $A(p\to r)\lesssim 1 $ for $1\leq r\leq p\leq \infty.$ Thus, by duality and the interpolation theorem, it will be enough to prove that 
\begin{equation}\label{rema} \|f\ast d\sigma\|_{L^{d-1}(\mathbb F_q^d, dx)} \lesssim \|f\|_{L^{d(d-1)/(d^2-2d+2)}(\mathbb F_q^d, dx)}~~\mbox{for all functions}~~f ~~\mbox{on}~~\mathbb F_q^d.\end{equation}
The authors in \cite{KS11} showed that this inequality holds for all characteristic functions on subsets of $\mathbb F_q^d.$ 
Here, we improve upon their work by obtaining the strong-type estimate.
\end{remark}

\subsection{Outline of  this paper} In the remaining parts of this paper, we focus on providing the detail proof of Theorem \ref{mainthm}.
In Section 2,  we review  the Fourier analysis in finite fields and prove key lemmas which are essential in proving our main theorem.
The proof of Theorem \ref{mainthm} for even dimensions $d\geq 6$ will be completed in Section 3.  Namely,  when $d\geq 6$ is any even integer, the inequality (\ref{rema}) will be proved in Section 3.
In the final section, we finish the proof of Theorem \ref{mainthm} by proving the inequality (\ref{rema}) for $d=4.$

\section{key lemmas}
In this section we drive key lemmas which play a crucial role in proving Theorem \ref{mainthm}.
We begin by reviewing the Discrete Fourier analysis and readers can be reffered to \cite{KS11} for more information on it.
Let $\mathbb F_q$ be a finite field with $q$ elements. Throughout this paper, we assume that $q$ is a power of odd prime so that the characteristic of $\mathbb F_q$ is greater than two. We denote by $\chi$ a  nontrivial additive character of $\mathbb F_q.$
Recall that  the orthogonality relation of the canonical additive character $\chi$ says that
$$ \sum_{x\in \mathbb F_q^d} \chi(m\cdot x)=\left\{ \begin{array}{ll} 0  \quad&\mbox{if}~~ m\neq (0,\dots,0)\\
q^d  \quad &\mbox{if} ~~m= (0,\dots,0), \end{array}\right.$$
where $\mathbb F_q^d$ denotes the $d$-dimensional vector space over $\mathbb F_q$ and $m\cdot x$ is the usual dot-product notation.
Denote by $(\mathbb F_q^d, dx)$ the vector space over $\mathbb F_q$, endowed with the normalized counting measure $``dx".$
Its dual space will be denoted by $(\mathbb F_q^d, dm)$ and we endow it with a counting measure $``dm".$
If $f: (\mathbb F_q^d, dx) \rightarrow \mathbb C,$ then the Fourier transform of the function $f$ is defined on $(\mathbb F_q^d, dm)$:
$$ \widehat{f}(m)=\int_{\mathbb F_q^d} f(x) \chi(-x\cdot m)~dx=  \frac{1}{q^d} \sum_{x\in \mathbb F_q^d} f(x) \chi(-x\cdot m)\quad \mbox{for}~~m\in \mathbb F_q^d.$$
We also recall  the Plancherel theorem:
$$ \sum_{m\in \mathbb F_q^d} |\widehat{f}(m)|^2 = \frac{1}{q^d} \sum_{x\in \mathbb F_q^d} |f(x)|^2.$$
Throughout this paper, we identify the set $E \subset \mathbb F_q^d$ with the characteristic function on the set $E.$
We denote by $(d\sigma)^{\vee}$ the inverse Fourier transform of the normalized surface measure $d\sigma$ on $S$  in (\ref{defv}) .
Recall that
$$ (d\sigma)^{\vee}(m)=\int_{S}\chi(m\cdot x) ~d\sigma(x) = \frac{1}{|S|} \sum_{x\in S} \chi(m\cdot x). $$

\subsection{Gauss sums and estimates of $(d\sigma)^{\vee}$}
Let $\eta$  denote  the quadratic character of ${\mathbb F_q}.$ For each $t\in {\mathbb F_q},$ the Gauss sum $G_t(\eta,\chi)$ is defined by 
$$ G_t(\eta,\chi)=\sum_{s\in {\mathbb F_q}\setminus \{0\}} \eta(s)\chi(ts).$$
The absolute value  of the Gauss sum is given as follows (see \cite{LN97, IKow04}):
$$|G_t(\eta,\chi)|=\left\{\begin{array}{ll} q^{\frac{1}{2}} \quad &\mbox{if}~~t\neq 0\\
0\quad&\mbox{if} ~~t=0,\end{array}\right..$$
It turns out that the inverse Fourier transform of $d\sigma$ can be written in terms of   the Gauss sum.
The following  was given in Lemma 4.1 of \cite{KS11}.

\begin{lemma} \label{Fdecay} Let $S$ be the variety in $\mathbb F_q^d$ as defined  in (\ref{defv}), and let $d\sigma$ be the normalized surface measure on $S.$
If $d\geq 2$ is even, then we have
 $$ (d\sigma)^\vee(m)=\left\{ \begin{array}{ll} q^{d-1} |S|^{-1}+ \frac{G_1^d}{|S|} (1-q^{-1})\eta(a_1\cdots a_d) \quad &\mbox{if}~~ m=(0,\dots,0)\\
\frac{G_1^d}{|S|}(1-q^{-1})\eta(a_1\cdots a_d)  \quad &\mbox{if}~~ m\neq (0,\dots,0),~ \frac{m_1^2}{a_1}+\cdots+ \frac{m_d^2}{a_d}=0\\
 -\frac{G_1^{d}}{q|S|} \eta(a_1 \cdots  a_d)  \quad &\mbox{if} ~~ \frac{m_1^2}{a_1}+\cdots+ \frac{m_d^2}{a_d}\neq 0. \end{array}\right.$$
Here and throughout this paper we denote by $\eta$  the quadratic character of ${\mathbb F_q}$  and we write $G_1$ for the Gauss sum  $G_1(\eta,\chi).$
\end{lemma}


 Lemma \ref{Fdecay} yields the following corollary. 

\begin{corollary}\label{decay}Let $S$ be the variety in $\mathbb F_q^d$ as defined  in (\ref{defv}) and let $d\sigma$ be the normalized surface measure on $S.$
If $d\geq 4$ is even, then we have
\begin{equation}\label{sizeS} |S|=q^{d-1}+ G_1^d (1-q^{-1}) \eta(a_1\cdots a_d) \sim q^{d-1},\end{equation}
and
\begin{equation}\label{evendecay}|(d\sigma)^\vee(m)|\sim\left\{ \begin{array}{ll} q^{-\frac{(d-2)}{2}} \quad &\mbox{if}~~ m\neq (0,\dots,0),~ \frac{m_1^2}{a_1}+\cdots+ \frac{m_d^2}{a_d}=0\\
                                                                                                                                     q^{-\frac{d}{2}} \quad &\mbox{if} ~~ \frac{m_1^2}{a_1}+\cdots+ \frac{m_d^2}{a_d}\neq 0.\end{array}\right.\end{equation}
\end{corollary} 
\begin{proof} By the definition of $(d\sigma)^\vee(0,\ldots,0)$, it is clear that $ (d\sigma)^\vee(0,\ldots,0)=1.$ Comparison with Lemma \ref{Fdecay} shows that
$$ |S|=q^{d-1}+ G_1^d (1-q^{-1})  \eta(a_1\cdots a_d).$$
Since $|G_1|=q^{1/2}$, it follows that $|S|\sim q^{d-1} $ for $d\geq 4$ even. This proves (\ref{sizeS}). 
The inequality (\ref{evendecay}) follows immediately from Lemma \ref{Fdecay}  because $|G_1|=q^{1/2}$ and $|S|\sim q^{d-1}$ for $d\geq 4$ even.
\end{proof}
\begin{remark} It is clear from (\ref{evendecay}) that  if $d\geq 4$ is even and $S$ is any nondegenerate quadratic surface in $\mathbb F_q^d,$ then
\begin{equation}\label{ko} |(d\sigma)^\vee(m)|=\left|\frac{1}{|S|} \sum_{x\in S} \chi(m\cdot x)\right| \sim \frac{1}{q^{d-1}}\left| \sum_{x\in S} \chi(m\cdot x)\right| \lesssim q^{-\frac{(d-2)}{2}} ~~\mbox{for} ~~m \in \mathbb F_q^d \setminus \{(0,\dots, 0)\}.\end{equation}
\end{remark}

\subsection{ Bochner-Riesz kernel}
Recall that $d\sigma$ is the normalized surface measure on the nondegenerate quadratic surface $S$ . 
In the finite field setting, the Bochner-Riesz kernel $K$  is a function on $(\mathbb F_q^d, dm)$ and it satisfies that  $ K = (d\sigma)^\vee -\delta_0.$
Recall that $dm$ denotes the counting measure on $\mathbb F_q^d.$
Notice that $K(m)=0$ if $m=(0,\ldots,0)$, and $ K(m)= (d\sigma)^\vee(m)$ otherwise.
Also observe that 
$$d\sigma = \widehat{K} +\widehat{\delta_0}= \widehat{K}+ 1.$$
Here, the last equality follows because  $\delta_0$ is defined on the vector space with the  counting measure $dm,$ and  its Fourier transform $\widehat{\delta_0}$ is defined on the dual space with the normalized counting measure $dx.$   More precisely, if $x\in (\mathbb F_q^d, dx),$ then 
$$\widehat{\delta_0}(x) =\int_{m\in \mathbb F_q^d} \chi(-m\cdot x) \delta_0(m)~dm= \sum_{m\in \mathbb F_q^d} \chi(-m\cdot x) \delta_0(m) = 1.$$ 
Our main lemma  is as follows.

\begin{lemma}\label{mainlem} Suppose that $d\geq 6$ is even. Then, for every $E\subset \mathbb F_q^d,$  we have
\begin{equation}\label{six} \|E\ast \widehat{K}\|_{L^{\frac{d-1}{2}}(\mathbb F_q^d,dx)}\lesssim \left\{\begin{array}{ll} 
q^{\frac{-d^2+2d-3}{d-1}} |E|^{\frac{d-3}{d-1}}\quad &\mbox{if} ~~ 1\leq |E|\leq q^{\frac{d-2}{2}}\\
q^{\frac{-d^2+d-1}{d-1}}|E|\quad &\mbox{if} ~~q^{\frac{d-2}{2}}\leq|E|\leq q^{\frac{d}{2}}\\
 q^{-d+1}|E|^{\frac{d-3}{d-1}} \quad &\mbox{if} ~~ q^{\frac{d}{2}}\leq |E|\leq q^d.\end{array}\right.\end{equation}
 where $K$ is the Bochner-Riesz kernel. On the other hand,  for every  $E\subset \mathbb F_q^4,$ it follows that
\begin{equation}\label{four}\|E\ast \widehat{K}\|_{L^{6}(\mathbb F_q^4,dx)}\lesssim \left\{\begin{array}{ll} 
q^{-\frac{19}{6}} |E|^{\frac{5}{6}}\quad &\mbox{if} ~~ 1\leq |E|\leq q\\
q^{-\frac{10}{3}}|E|\quad &\mbox{if} ~~q\leq|E|\leq q^2\\
q^{-3}|E|^{\frac{5}{6}} \quad &\mbox{if} ~~ q^2\leq |E|\leq q^4.\end{array}\right.\end{equation}
\end{lemma}

\begin{proof} Using the interpolation theorem, it suffices to prove that  the following two inequalities hold for all d$\geq 4$ even:

\begin{equation}\label{samil}
\|E\ast \widehat{K}\|_{L^{\infty}(\mathbb F_q^d,dx)} \lesssim  q^{-d+1} |E|
\end{equation}
and 
\begin{equation}\label{sami}\|E\ast \widehat{K}\|_{L^{2}(\mathbb F_q^d,dx)}\lesssim \left\{\begin{array}{ll} 
q^{\frac{-2d+1}{2}} |E|^{\frac{1}{2}}\quad &\mbox{if} ~~ 1\leq |E|\leq q^{\frac{d-2}{2}}\\
q^{\frac{-5d+4}{4}}|E|\quad &\mbox{if} ~~q^{\frac{d-2}{2}}\leq|E|\leq q^{\frac{d}{2}}\\
q^{-d+1}|E|^{\frac{1}{2}} \quad &\mbox{if} ~~ q^{\frac{d}{2}}\leq |E|\leq q^d.\end{array}\right.\end{equation}
 
The estimate (\ref{samil}) can be obtained by applying  Young's inequality. In fact, we see that
$$\|E\ast \widehat{K}\|_{L^{\infty}(\mathbb F_q^d,dx)} \leq \|\widehat{K}\|_{L^\infty(\mathbb F_q^d, dx)} \|E\|_{L^1(\mathbb F_q^d,dx)}.$$
Since $\|\widehat{K}\|_{L^\infty(\mathbb F_q^d, dx)}\lesssim q$ and $\|E\|_{L^1(\mathbb F_q^d,dx)}=q^{-d}|E|,$ the inequality (\ref{samil}) is established. To prove the inequality (\ref{sami}), first use the Plancherel theorem. It follows that
$$\|E\ast \widehat{K}\|^2_{L^{2}(\mathbb F_q^d,dx)}= \|\widehat{E} K\|^2_{L^{2}(\mathbb F_q^d,dm)}.$$
Now, we recall that $dx$ is the normalized counting measure but $dm$ is the counting measure.
Thus,  the expression above is given by
\begin{align*}\sum_{m\in \mathbb F_q^d}|\widehat{E}(m)|^2|K(m)|^2 =&\sum_{m\neq (0,\dots,0)} |\widehat{E}(m)|^2|(d\sigma)^\vee(m)|^2\\
&\sim \frac{1}{q^{d-2}} \sum_{\substack{m\neq (0,\dots,0):\\ \frac{m_1^2}{a_1}+\cdots+ \frac{m_d^2}{a_d}=0}} |\widehat{E}(m)|^2 +\frac{1}{q^d}\sum_{\substack{m\neq (0,\dots,0):\\ \frac{m_1^2}{a_1}+\cdots+ \frac{m_d^2}{a_d}\neq 0}} |\widehat{E}(m)|^2 =\mbox{I} +\mbox{II},\end{align*}
where the first line and the second line follow from the definition of $K$ and the inequality (\ref{evendecay}) in Corollary \ref{decay}, respectively.
Applying the Plancherel theorem, it is clear that 
\begin{equation}\label{EII} \mbox{II} \leq \frac{1}{q^d}\sum_{m \in \mathbb F_q^d} |\widehat{E}(m)|^2 = q^{-2d}|E|.\end{equation}

In order to obtain a good upper bound of $\mbox{I}$, we shall conduct two different estimates on $\mbox{I}.$
First,  the Plancherel theorem yields 
\begin{equation}\label{EI1} \mbox{I} \leq \frac{1}{q^{d-2}} \sum_{m\in \mathbb F_q^d} |\widehat{E}(m)|^2 = \frac{|E|}{q^{2d-2}} .\end{equation}
On the other hand, it follows that
$$\mbox{I}\leq \frac{1}{q^{d-2}} \sum_{ \frac{m_1^2}{a_1}+\cdots+ \frac{m_d^2}{a_d}=0} |\widehat{E}(m)|^2
=\frac{1}{q^{3d-2}} \sum_{ \frac{m_1^2}{a_1}+\cdots+ \frac{m_d^2}{a_d}=0} \sum_{x,y\in E} \chi(-m\cdot (x-y)).$$
Now, let $S_a=\{m\in \mathbb F_q^d: \frac{m_1^2}{a_1}+\cdots+ \frac{m_d^2}{a_d}=0\}$ which is also a nondegenerate quadratic surface with $|S_a|\sim q^{d-1}.$ Then the expression above can be written by
$$ \frac{1}{q^{3d-2}} \sum_{x,y\in E:x=y} |S_a| + \frac{1}{q^{3d-2}} \sum_{x,y\in E:x\neq y} \left(\sum_{m\in S_a} \chi(-m\cdot (x-y)) \right). $$
Now, we see from (\ref{ko}) that if $x\neq y$, then $\left|\sum_{m\in S_a} \chi(-m\cdot (x-y)) \right|\lesssim q^{\frac{d}{2}}.$
Thus, we obtain that
$$ \mbox{I}\lesssim q^{-2d+1}|E| + q^{\frac{-5d+4}{2}} |E|^2.$$
Combining this with the inequality (\ref{EI1}) gives 
$$ \mbox{I} \lesssim \min \left( \frac{|E|}{q^{2d-2}},~ q^{-2d+1}|E| + q^{\frac{-5d+4}{2}} |E|^2 \right).$$
In conjunction with  the inequality (\ref{EII}), this shows that
$$ \|E\ast \widehat{K}\|^2_{L^{2}(\mathbb F_q^d,dx)} \lesssim  \min \left( \frac{|E|}{q^{2d-2}} ,~ q^{-2d+1}|E| + q^{\frac{-5d+4}{2}} |E|^2 \right)~+ q^{-2d}|E|.$$ 
Since $ (\alpha+\beta)^{1/2}\sim \alpha^{1/2} + \beta^{1/2} $ for $\alpha, \beta \geq 0$,  it also follows that
$$\|E\ast \widehat{K}\|_{L^{2}(\mathbb F_q^d,dx)} \lesssim  \min \left( q^{-d+1} |E|^{\frac{1}{2}},~ q^{\frac{-2d+1}{2}}|E|^{\frac{1}{2}} + q^{\frac{-5d+4}{4}} |E| \right)~+ q^{-d}|E|^{\frac{1}{2}}.$$ 
A direct computation shows that this implies  the inequality (\ref{sami}). We complete the proof of Lemma \ref{mainlem}.

\end{proof}
 
 \section{Proof of Theorem \ref{mainthm} for $d\geq 6$}
In this section we provide the complete proof of Theorem \ref{mainthm} in the case that $d\geq 6$ is even.
The proof for $d=4$ shall be independently given in the following section. The main reason is as follows.
Lemma \ref{mainlem} shall be used to prove Theorem \ref{mainthm}. 
If $d\geq 6$ is even, then we have seen that  the inequality (\ref{six}) of Lemma \ref{mainlem} follows by interpolating (\ref{samil}) and (\ref{sami}).
However, if $d$ is four, then such an interpolation is too meaningless to assert that  (\ref{six}) holds for $d=4.$
As an alternative approach, the inequality (\ref{four}) of Lemma \ref{mainlem} shall be applied to complete the proof for $d=4.$
In this case we need more delicate estimates.\\

Now we start proving Theorem \ref{mainthm} for $d\geq 6$ even.
As mentioned in Remark \ref{remgo}, it is enough to prove the following statement.
\begin{theorem}\label{res1}  Let $S$ be the variety in $\mathbb F_q^d$ as defined in (\ref{defv}). If $d\geq 6$ is even, then we have
$$ \|f\ast d\sigma\|_{L^r(\mathbb F_q^d, dx)} \lesssim \|f\|_{L^p(\mathbb F_q^d, dx)}~~\mbox{for all functions}~~f ~~\mbox{on}~~\mathbb F_q^d,$$
where $ (p,r)=\left(\frac{d^2-d}{d^2-2d+2}, d-1\right).$\end{theorem}
\begin{proof} Let $p=\frac{d^2-d}{d^2-2d+2}$ and $ r=d-1.$ 
We aim to prove that for every complex-valued function $f$ on $\mathbb F_q^d,$
$$\|f\ast d\sigma\|_{L^r(\mathbb F_q^d, dx)} \lesssim \|f\|_{L^p(\mathbb F_q^d, dx)} = \left(q^{-d} \sum_{x\in \mathbb F_q^d} |f(x)|^p\right)^{\frac{1}{p}}.$$
As in \cite{LL10} we proceed with the proof by decomposing the function $f$ to which the operator is applied into level sets.
Without loss of generality, we may assume that $f$ is a nonnegative real-valued function and 
$$ \sum_{x\in \mathbb F_q^d} |f(x)|^p =1.$$
Therefore, we may also assume that
\begin{equation}\label{simp} f= \sum_{k=0}^\infty 2^{-k} E_k ,\end{equation}
where $E_0, E_1, \dots$ are disjoint subsets of $\mathbb F_q^d.$ 
It follows from these assumptions that  
\begin{equation}\label{as1}\sum_{j=0}^\infty 2^{-pj}|E_j|=1,\end{equation}
and hence  for every $j=0,1,\dots,$
\begin{equation}\label{as2} |E_j|\leq 2^{pj}.\end{equation}
Recall that $d\sigma= \widehat{K} +1$ where $K$ is the  Bochner-Riesz kernel.
It follows that 
$$\|f\ast d\sigma\|_{L^r(\mathbb F_q^d, dx)} \leq \|f\ast \widehat{K}\|_{L^r(\mathbb F_q^d, dx)} + \|f\ast 1\|_{L^r(\mathbb F_q^d, dx)}.$$
Since $r>p$ and $dx$ is the normalized counting measure on $\mathbb F_q^d,$  it is clear from  Young's inequality that 
$$\|f\ast 1\|_{L^r(\mathbb F_q^d, dx)} \leq \|f\|_{L^p(\mathbb F_q^d, dx)}.$$
Therefore, it suffices to prove the following inequality
$$ \|f\ast \widehat{K}\|_{L^r(\mathbb F_q^d, dx)} \lesssim \|f\|_{L^p(\mathbb F_q^d, dx)}.$$
Since we have assumed that $\sum_{x\in \mathbb F_q^d} |f(x)|^p =1$, we see that $\|f\|_{L^p(\mathbb F_q^d, dx)}=q^{-d/p}.$
Also observe that 
$$ \|f\ast \widehat{K}\|^2_{L^r(\mathbb F_q^d, dx)} =\|(f\ast \widehat{K})(f\ast \widehat{K})\|_{L^{\frac{r}{2}}(\mathbb F_q^d, dx)}.$$
From these observations, our task is to show that
\begin{equation}\label{same} q^{\frac{2d}{p}} \|(f\ast \widehat{K})(f\ast \widehat{K})\|_{L^{\frac{r}{2}}(\mathbb F_q^d, dx)} \lesssim 1.\end{equation}
Using  (\ref{simp}), we see that
\begin{align*} q^{\frac{2d}{p}}\|(f\ast \widehat{K})(f\ast \widehat{K})\|_{L^{\frac{r}{2}}(\mathbb F_q^d, dx)} 
&\leq q^{\frac{2d}{p}}\sum_{k=0}^\infty\sum_{j=0}^\infty 2^{-k-j} \|(E_k\ast \widehat{K})(E_j\ast \widehat{K})\|_{L^{\frac{r}{2}}(\mathbb F_q^d, dx)}\\
&\lesssim  q^{\frac{2d}{p}} q^{-d+1} \sum_{k=0}^\infty\sum_{j=k}^\infty 2^{-k-j}|E_k|  \|(E_j\ast \widehat{K})\|_{L^{\frac{r}{2}}(\mathbb F_q^d, dx)},\end{align*}
where the last line follows by the symmetry of $k$ and $j,$ and the inequality (\ref{samil}). Now, for each $j=0,1,2,\dots,$
we consider the following three sets:
$$J_1=\{j: 1\leq |E_j|\leq q^{\frac{d-2}{2}}\}, $$
$$J_2=\{j: q^{\frac{d-2}{2}}< |E_j|\leq q^\frac{d}{2}\}, $$
and
$$J_3=\{j: q^{\frac{d}{2}}< |E_j|\leq q^d\}.$$
Since $r/2=(d-1)/2,$ it is clear from  (\ref{six}) in Lemma \ref{mainlem} that our goal is to prove the following three inequalities:
\begin{equation}\label{G1} A_1:= q^{\frac{2d}{p}} q^{-d+1} q^{\frac{-d^2+2d-3}{d-1}} \sum_{k=0}^\infty\sum_{j=k: j\in J_1}^\infty 2^{-k-j}|E_k||E_j|^{\frac{d-3}{d-1}} \lesssim 1,\end{equation}
\begin{equation}\label{G2} A_2:= q^{\frac{2d}{p}} q^{-d+1}q^{\frac{-d^2+d-1}{d-1}}\sum_{k=0}^\infty\sum_{j=k: j\in J_2}^\infty 2^{-k-j}|E_k||E_j|\lesssim 1,\end{equation}
\begin{equation}\label{G3} A_3:= q^{\frac{2d}{p}} q^{-d+1}q^{-d+1}\sum_{k=0}^\infty\sum_{j=k: j\in J_3}^\infty 2^{-k-j}|E_k||E_j|^{\frac{d-3}{d-1}}\lesssim 1.\end{equation}
First, we prove that the inequality (\ref{G1}) holds. Since $p=\frac{d^2-d}{d^2-2d+2},$ a direct computation shows that $q^{\frac{2d}{p}} q^{-d+1} q^{\frac{-d^2+2d-3}{d-1}}=1.$ Now recall from (\ref{as2}) that $|E_j|\leq 2^{pj}$ for all $j=0,1,\dots.$ Therefore, it follows that
$$ A_1\leq \sum_{k=0}^\infty\sum_{j=k: j\in J_1}^\infty 2^{-k-j}|E_k|2^{\frac{jp(d-3)}{d-1}} 
\leq \sum_{k=0}^\infty2^{-k}|E_k|\sum_{j=k}^\infty 2^{j\left(-1+\frac{p(d-3)}{d-1}\right)}.$$
Since $-1+\frac{p(d-3)}{d-1}= \frac{-d-2}{d^2-2d+2} <0$ and the sum over $j$ is a geometric series, we see that $\sum_{j=k}^\infty 2^{j\left(-1+\frac{p(d-3)}{d-1}\right)} \sim 2^{k\left(-1+\frac{p(d-3)}{d-1}\right)}.$ Thus, the inequality (\ref{G1}) is established as follows: 
$$ A_1\lesssim \sum_{k=0}^\infty |E_k|2^{k\left(-2+\frac{p(d-3)}{d-1}\right)} \leq \sum_{k=0}^\infty |E_k|2^{-pk} =1, $$
where we used the simple observation that  $-2+\frac{p(d-3)}{d-1}< -p,$ and then the assumption $(\ref{as1}).$ \\
Second, we prove that the inequality (\ref{G2}) holds. Let $\varepsilon =\frac{2d-4}{d^2-d}.$ Since $d\geq 6,$ we see that $0<\varepsilon <1.$ Write $A_2$ as follows:
$$ A_2= q^{\frac{2d}{p}} q^{-d+1}q^{\frac{-d^2+d-1}{d-1}}\sum_{k=0}^\infty\sum_{j=k: j\in J_2}^\infty 2^{-k-j}|E_k||E_j|^{1-\varepsilon} |E_j|^\varepsilon.$$
Since $0<\varepsilon <1,$ we notice from (\ref{as2}) that $|E_j|^{1-\varepsilon}\leq 2^{p(1-\varepsilon)j}.$ By the definition of the set $J_2$, we also see that 
$|E_j|^\varepsilon \leq q^\frac{d\varepsilon}{2}$ for all $j\in J_2.$  Then, we have
\begin{align*} A_2&\leq q^{\frac{2d}{p}} q^{-d+1}q^{\frac{-d^2+d-1}{d-1}} q^\frac{d\varepsilon}{2}\sum_{k=0}^\infty\sum_{j=k: j\in J_2}^\infty 2^{-k-j}|E_k|2^{p(1-\varepsilon)j}\\
&\leq  q^{\frac{2d}{p}} q^{-d+1}q^{\frac{-d^2+d-1}{d-1}}q^\frac{d\varepsilon}{2} \sum_{k=0}^\infty 2^{-k}|E_k|\sum_{j=k}^\infty 2^{j\left(-1+p(1-\varepsilon)\right)}.\end{align*}
Notice that $q^{\frac{2d}{p}} q^{-d+1}q^{\frac{-d^2+d-1}{d-1}}q^\frac{d\varepsilon}{2}=1,$ and the geometric series over $j$ converges to $\sim 2^{k\left(-1+p(1-\varepsilon)\right)}$ because  $ -1+p(1-\varepsilon) = \frac{-d+2}{d^2-2d+2}<0$ for $d\geq 6.$ From this observation and (\ref{as1}), the inequality (\ref{G2}) follows because we have
$$ A_2\lesssim \sum_{k=0}^\infty |E_k| 2^{k\left(-2+p(1-\varepsilon)\right)}= \sum_{k=0}^\infty |E_k| 2^{-pk} =1.$$
Finally, we show that the inequality (\ref{G3}) holds. As in the proof of the inequality (\ref{G2}),
we  let $0<\delta=\frac{4}{d^2-d}<1$ for $d\geq 6.$ The value $A_3$ is written by
$$ A_3=q^{\frac{2d}{p}} q^{-d+1}q^{-d+1}\sum_{k=0}^\infty\sum_{j=k: j\in J_3}^\infty 2^{-k-j}|E_k||E_j|^{\delta+\frac{d-3}{d-1}}|E_j|^{-\delta}.$$
Notice  from (\ref{as2}) that $ |E_j|^{\delta+\frac{d-3}{d-1}} \leq 2^{p\left(\delta+\frac{d-3}{d-1}\right)j}$ for all $j=0,1,2,\dots.$
By the definition of  $J_3$, it is easy to notice that $|E_j|^{-\delta} \leq q^{\frac{-d\delta}{2}}$ for $j\in J_3.$ It therefore follows that

\begin{align*} A_3&\leq q^{\frac{2d}{p}} q^{-d+1}q^{-d+1}q^{\frac{-d\delta}{2}}\sum_{k=0}^\infty\sum_{j=k: j\in J_3}^\infty 2^{-k-j}|E_k|2^{p\left(\delta+\frac{d-3}{d-1}\right)j}\\
&\leq \sum_{k=0}^\infty 2^{-k}|E_k| \sum_{j=k}^\infty 2^{\left(-1+p\left(\delta+\frac{d-3}{d-1}\right)\right)j}\\
&\sim \sum_{k=0}^\infty|E_k|2^{\left(-2+p\left(\delta+\frac{d-3}{d-1}\right)\right)k} = \sum_{k=0}^\infty|E_k| 2^{-pk}=1,
\end{align*}
where we used the facts that $q^{\frac{2d}{p}} q^{-d+1}q^{-d+1}q^{\frac{-d\delta}{2}}=1,$ $\left(-1+p\left(\delta+\frac{d-3}{d-1}\right)\right)=\frac{-d+2}{d^2-2d+2}<0$ for $d\geq 6,$ and $\left(-2+p\left(\delta+\frac{d-3}{d-1}\right)\right)=-p,$ and then the assumption (\ref{as1}) for the last equality. Thus, the inequality (\ref{G3}) holds and the proof of Theorem \ref{res1} is complete.\end{proof}

 \section{Proof of Theorem \ref{mainthm} for $d=4$}

As observed in Remark \ref{remgo}, it amounts to showing the following statement.
\begin{theorem}\label{res2}  Let $S$ be the variety in $\mathbb F_q^4$ as defined  in (\ref{defv}). Then, we have
$$ \|f\ast d\sigma\|_{L^3(\mathbb F_q^4, dx)} \lesssim \|f\|_{L^{\frac{6}{5}}(\mathbb F_q^4, dx)}~~\mbox{for all functions}~~f ~~\mbox{on}~~\mathbb F_q^4.$$
\end{theorem}
\begin{proof} We will proceed by the similar ways as in the previous section.
However, the proof of the theorem will be based on (\ref{four}), rather than (\ref{six}) in Lemma \ref{mainlem}. We begin by recalling from  (\ref{four}) and (\ref{sami}) that
\begin{equation}\label{6sam}\|E\ast \widehat{K}\|_{L^{6}(\mathbb F_q^4,dx)}\lesssim \left\{\begin{array}{ll} 
q^{-\frac{19}{6}} |E|^{\frac{5}{6}}\quad &\mbox{if} ~~ 1\leq |E|\leq q\\
q^{-\frac{10}{3}}|E|\quad &\mbox{if} ~~q\leq|E|\leq q^2\\
q^{-3}|E|^{\frac{5}{6}} \quad &\mbox{if} ~~ q^2\leq |E|\leq q^4,\end{array}\right.\end{equation}
and
\begin{equation}\label{sam}\|E\ast \widehat{K}\|_{L^{2}(\mathbb F_q^4,dx)}\lesssim \left\{\begin{array}{ll} 
q^{-\frac{7}{2}} |E|^{\frac{1}{2}}\quad &\mbox{if} ~~ 1\leq |E|\leq q\\
q^{-4}|E|\quad &\mbox{if} ~~q\leq|E|\leq q^2\\
q^{-3}|E|^{\frac{1}{2}} \quad &\mbox{if} ~~ q^2\leq |E|\leq q^4.\end{array}\right.\end{equation}

We must show that for all complex-valued functions $f$ on $\mathbb F_q^4,$ 
$$\|f\ast d\sigma\|_{L^3(\mathbb F_q^4, dx)} \lesssim \|f\|_{L^{\frac{6}{5}}(\mathbb F_q^4, dx)}.$$
As noticed in the previous section, it suffices to prove this inequality under the following assumptions:
$$ \sum_{x\in \mathbb F_q^4}|f(x)|^\frac{6}{5} =1 \quad\mbox{and}\quad f=\sum_{k=0}^\infty 2^{-k} E_k,$$
where $E_0, E_1, \dots$ are disjoint subsets of $\mathbb F_q^4.$ From these assumptions, it is clear that 
\begin{equation}\label{one}
\sum_{j=0}^\infty 2^{-\frac{6j}{5}} |E_j|=1 \quad\mbox{for all}~j=0,1,\dots.
\end{equation}
This clearly implies that
\begin{equation}\label{size} |E_j|\leq 2^{\frac{6j}{5}} \quad\mbox{for all}~j=0,1,\dots.\end{equation}
According to (\ref{same}), it is enough to prove that 
$$ q^{\frac{20}{3}} \|(f\ast \widehat{K})(f\ast \widehat{K})\|_{L^{\frac{3}{2}}(\mathbb F_q^4, dx)} \lesssim 1.$$
Since $f=\sum_{k=0}^\infty 2^{-k} E_k,$ it is enough to show that
$$ q^{\frac{20}{3}} \sum_{k=0}^\infty\sum_{j=0}^\infty 2^{-k-j} \|(E_k\ast \widehat{K})(E_j\ast \widehat{K})\|_{L^{\frac{3}{2}}(\mathbb F_q^4, dx)}\lesssim 1.$$ By the symmetry of $k$ and $j,$ and the H\" older inequality, our task is to prove 
$$q^{\frac{20}{3}}\sum_{k=0}^\infty\sum_{j=k}^\infty 2^{-k-j} 
 \|(E_k\ast \widehat{K})\|_{L^6(\mathbb F_q^4,dx)} \|(E_j\ast \widehat{K})\|_{L^{2}(\mathbb F_q^4, dx)}\lesssim 1.$$
Main steps to prove this inequality are summarized as follows.
By considering the sizes of $|E_k|$ and $|E_j|$, we first decompose  $\sum_{k=0}^\infty\sum_{j=k}^\infty $ as nine parts. Next, using the estimates (\ref{6sam}),(\ref{sam}), (\ref{size}), (\ref{one}), and a convergence property of a geometric series,  we  show that each part of them is $\lesssim 1,$ which completes the proof of Theorem \ref{res2}. For the sake of completeness, we shall give full details. To do this, let us define the following $9$ sets: for $N=\{0,1,\dots \},$
$$I_1=\{(k,j)\in N\times N: k\leq j,~ 1\leq |E_k| \leq q,~ 1\leq |E_j|\leq q\},$$
$$I_2=\{(k,j)\in N\times N: k\leq j,~ 1\leq |E_k|\leq q, ~q<|E_j|\leq q^2\},$$
$$I_3=\{(k,j)\in N\times N: k\leq j,~ 1\leq |E_k|\leq q,~ q^2<|E_j|\leq q^4\},$$
$$I_4=\{(k,j)\in N\times N: k\leq j,~ q< |E_k|\leq q^2,~ 1\leq|E_j|\leq q\},$$
$$I_5=\{(k,j)\in N\times N: k\leq j,~ q< |E_k|\leq q^2,~ q<|E_j|\leq q^2\},$$
$$I_6=\{(k,j)\in N\times N: k\leq j,~ q< |E_k|\leq q^2,~ q^2<|E_j|\leq q^4\},$$
$$I_7=\{(k,j)\in N\times N: k\leq j,~ q^2< |E_k|\leq q^4,~ 1\leq|E_j|\leq q\},$$
$$I_8=\{(k,j)\in N\times N: k\leq j,~ q^2< |E_k|\leq q^4, ~q<|E_j|\leq q^2\},$$
$$I_9=\{(k,j)\in N\times N: k\leq j,~ q^2< |E_k|\leq q^4,~ q^2<|E_j|\leq q^4\}.$$
\subsection{ Estimate of the sum over $I_1$} It follows from (\ref{6sam}) and (\ref{sam}) that
\begin{align*}&q^{\frac{20}{3}}\sum_{(k,j)\in I_1} 2^{-k-j}  \|(E_k\ast \widehat{K})\|_{L^6(\mathbb F_q^4,dx)} \|(E_j\ast \widehat{K})\|_{L^{2}(\mathbb F_q^4, dx)}\\
\lesssim& \sum_{(k,j)\in I_1} 2^{-k-j}|E_k|^{\frac{5}{6}} |E_j|^{\frac{1}{2}}
\leq \sum_{(k,j)\in I_1} 2^{-k-j}|E_k| 2^{\frac{3j}{5}} \quad\mbox{since} \quad|E_j|^{\frac{1}{2}} \leq 2^{\frac{3j}{5}} \quad\mbox{by (\ref{size})}\\
\leq& \sum_{k=0}^\infty |E_k|2^{-k} \sum_{j=k}^\infty 2^{-\frac{2j}{5}} \sim \sum_{k=0}^\infty |E_k| 2^{-\frac{7k}{5}}
\leq \sum_{k=0}^\infty 2^{-\frac{6k}{5}}|E_k| = 1 \quad\mbox{by (\ref{one})}.\end{align*}

\subsection{Estimate of the sum over $I_2$} It follows from (\ref{6sam}) and (\ref{sam}) that
\begin{align*}&q^{\frac{20}{3}}\sum_{(k,j)\in I_2} 2^{-k-j}  \|(E_k\ast \widehat{K})\|_{L^6(\mathbb F_q^4,dx)} \|(E_j\ast \widehat{K})\|_{L^{2}(\mathbb F_q^4, dx)}\\
\lesssim& q^{-\frac{1}{2}} \sum_{(k,j)\in I_2} 2^{-k-j}|E_k|^{\frac{5}{6}} |E_j| \leq q^{-\frac{1}{2}}\sum_{(k,j)\in I_2} 2^{-j}|E_j|\quad \mbox{since}\quad 2^{-k}|E_k|^{\frac{5}{6}}\leq 1 \quad\mbox{by (\ref{size})}\\
\leq&  \sum_{(k,j)\in I_2}2^{-j}|E_j|^{\frac{3}{4}} \quad\mbox{since}\quad |E_j|^{\frac{1}{4}} \leq q^{\frac{1}{2}} \quad\mbox{for}~~ (k,j)\in I_2\\
\leq&\sum_{k=0}^\infty \sum_{j=k}^\infty 2^{-\frac{j}{10}} \quad\mbox{by (\ref{size})}\\
 \sim &\sum_{k=0}^\infty 2^{-\frac{k}{10}} \sim 1. \end{align*}

\subsection{Estimate of the sum over $I_3$} It follows from (\ref{6sam}) and (\ref{sam}) that
\begin{align*}&q^{\frac{20}{3}}\sum_{(k,j)\in I_3} 2^{-k-j}  \|(E_k\ast \widehat{K})\|_{L^6(\mathbb F_q^4,dx)} \|(E_j\ast \widehat{K})\|_{L^{2}(\mathbb F_q^4, dx)}\\
\lesssim&q^{\frac{1}{2}} \sum_{(k,j)\in I_3}  2^{-k-j} |E_k|^{\frac{5}{6}} |E_j|^{\frac{1}{2}}\leq q^{\frac{1}{2}}\sum_{(k,j)\in I_3}  2^{-j}|E_j|^{\frac{1}{2}} \quad\mbox{by (\ref{size})}\\
= & \sum_{(k,j)\in I_3} 2^{-j}|E_j|^{\frac{3}{4}}q^{\frac{1}{2}} |E_j|^{-\frac{1}{4}}<\sum_{(k,j)\in I_3} 2^{-j}|E_j|^{\frac{3}{4}}\quad\mbox{since}\quad q^2<|E_j| ~~\mbox{for} ~~(k,j)\in I_3\\
\leq& \sum_{k=0}^\infty \sum_{j=k}^\infty 2^{-\frac{j}{10}}\sim \sum_{k=0}^\infty2^{-\frac{k}{10}} \sim 1. \end{align*}
where  (\ref{size}) was also used to obtain the last inequality .

\subsection{ Estimate of the sum over $I_4$} It follows from (\ref{6sam}) and (\ref{sam}) that
\begin{align*}&q^{\frac{20}{3}}\sum_{(k,j)\in I_4} 2^{-k-j}  \|(E_k\ast \widehat{K})\|_{L^6(\mathbb F_q^4,dx)} \|(E_j\ast \widehat{K})\|_{L^{2}(\mathbb F_q^4, dx)}\\
\lesssim&q^{-\frac{1}{6}} \sum_{(k,j)\in I_4}  2^{-k-j} |E_k| |E_j|^{\frac{1}{2}}\leq q^{-\frac{1}{6}}\sum_{(k,j)\in I_4} 2^{-k}|E_k| 2^{-j}2^{\frac{3j}{5}}\quad \mbox{by (\ref{size})}\\
\leq & q^{-\frac{1}{6}}\sum_{k=0}^\infty 2^{-k} |E_k| \sum_{j=k}^\infty 2^{-\frac{2j}{5}}\sim q^{-\frac{1}{6}}\sum_{k=0}^\infty 2^{-\frac{7k}{5}}|E_k|\\
\leq& \sum_{k=0}^\infty 2^{-\frac{6k}{5}}|E_k| =1\quad\mbox{by (\ref{one})}.\end{align*}

\subsection{ Estimate of the sum over $I_5$} It follows from (\ref{6sam}) and (\ref{sam}) that
\begin{align*}&q^{\frac{20}{3}}\sum_{(k,j)\in I_5} 2^{-k-j}  \|(E_k\ast \widehat{K})\|_{L^6(\mathbb F_q^4,dx)} \|(E_j\ast \widehat{K})\|_{L^{2}(\mathbb F_q^4, dx)}\\
\lesssim&q^{-\frac{2}{3}} \sum_{(k,j)\in I_5}  2^{-k-j} |E_k| |E_j|\leq \sum_{(k,j)\in I_5} 2^{-k-j}|E_k| |E_j|^{\frac{2}{3}}\quad \mbox{since}~~ |E_j|^{\frac{1}{3}}\leq q^{\frac{2}{3}}  \quad\mbox{for}~~ (k,j)\in I_5 \\
\leq &\sum_{k=0}^\infty 2^{-k} |E_k| \sum_{j=k}^\infty 2^{-\frac{j}{5}}\sim \sum_{k=0}^\infty 2^{-\frac{6k}{5}}|E_k|=1,
\end{align*}
where we used (\ref{size}), the convergence of a geometric series, and (\ref{one}) in the last line.

\subsection{ Estimate of the sum over $I_6$} 
It follows from (\ref{6sam}) and (\ref{sam}) that
\begin{align*}&q^{\frac{20}{3}}\sum_{(k,j)\in I_6} 2^{-k-j}  \|(E_k\ast \widehat{K})\|_{L^6(\mathbb F_q^4,dx)} \|(E_j\ast \widehat{K})\|_{L^{2}(\mathbb F_q^4, dx)}\\
\lesssim&q^{\frac{1}{3}} \sum_{(k,j)\in I_6}  2^{-k-j} |E_k| |E_j|^{\frac{1}{2}}
< \sum_{(k,j)\in I_6} 2^{-k-j}|E_k| |E_j|^{\frac{1}{2} +\frac{1}{6}}\quad \mbox{since}~~ |E_j|^{-\frac{1}{6}}< q^{-\frac{1}{3}}  \quad\mbox{for}~~ (k,j)\in I_6 \\
\leq &\sum_{k=0}^\infty 2^{-k} |E_k| \sum_{j=k}^\infty 2^{-\frac{j}{5}} \sim \sum_{k=0}^\infty  2^{-\frac{6k}{5}}|E_k|=1,
\end{align*}
where (\ref{size}), the convergence of a geometric series, and (\ref{one}) were also applied in the last line.

\subsection{ Estimate of the sum over $I_7$} It follows from (\ref{6sam}) and (\ref{sam}) that
\begin{align*}&q^{\frac{20}{3}}\sum_{(k,j)\in I_7} 2^{-k-j}  \|(E_k\ast \widehat{K})\|_{L^6(\mathbb F_q^4,dx)} \|(E_j\ast \widehat{K})\|_{L^{2}(\mathbb F_q^4, dx)}\\
\lesssim&q^{\frac{1}{6}} \sum_{(k,j)\in I_7}  2^{-k-j} |E_k|^{\frac{5}{6}} |E_j|^{\frac{1}{2}}= q^{\frac{1}{6}} \sum_{(k,j)\in I_7}  2^{-k-j} |E_k||E_k|^{-\frac{1}{6}} |E_j|^{\frac{1}{2}}\\
<& q^{\frac{1}{6} -\frac{1}{3}} \sum_{(k,j)\in I_7} 2^{-k}|E_k| 2^{-\frac{2j}{5}} \quad\mbox{by observing}
\quad  |E_k|^{-\frac{1}{6}}< q^{-\frac{1}{3}} \quad{for}~~(k,j)\in I_7 ~~ \mbox{and by (\ref{size})}\\
\leq& \sum_{k=0}^\infty 2^{-k} |E_k| \sum_{j=k}^\infty 2^{-\frac{2j}{5}} \sim \sum_{k=0}^\infty  2^{-\frac{7k}{5}}|E_k|\leq \sum_{k=0}^\infty  2^{-\frac{6k}{5}}|E_k|=1.
\end{align*}

\subsection{ Estimate of the sum over $I_8$} It follows from (\ref{6sam}) and (\ref{sam}) that
\begin{align*} & q^{\frac{20}{3}}\sum_{(k,j)\in I_8} 2^{-k-j}  \|(E_k\ast \widehat{K})\|_{L^6(\mathbb F_q^4,dx)} 
\|(E_j\ast \widehat{K})\|_{L^{2}(\mathbb F_q^4, dx)}\\
\lesssim &q^{-\frac{1}{3}} \sum_{(k,j)\in I_8}  2^{-k-j} |E_k|^{\frac{5}{6}} |E_j|= q^{-\frac{1}{3}} \sum_{(k,j)\in I_8}  2^{-k-j} |E_k||E_j|^{\frac{2}{3}} |E_k|^{-\frac{1}{6}}|E_j|^{\frac{1}{3}}\\
<& \sum_{(k,j)\in I_8}2^{-k-j}|E_k||E_j|^{\frac{2}{3}} \quad\mbox{since}~~ |E_k|^{-\frac{1}{6}} < q^{-\frac{1}{3}} ,\quad |E_j|^{\frac{1}{3}} \leq q^{\frac{2}{3}} ~~ \mbox{for}~~(k,j)\in I_8\\
\leq& \sum_{k=0}^\infty 2^{-k} |E_k| \sum_{j=k}^\infty 2^{-\frac{j}{5}} \sim \sum_{k=0}^\infty  2^{-\frac{6k}{5}}|E_k|=1.
\end{align*}
\subsection{ Estimate of the sum over $I_9$} It follows from (\ref{6sam}) and (\ref{sam}) that
\begin{align*} & q^{\frac{20}{3}}\sum_{(k,j)\in I_9} 2^{-k-j}  \|(E_k\ast \widehat{K})\|_{L^6(\mathbb F_q^4,dx)} 
\|(E_j\ast \widehat{K})\|_{L^{2}(\mathbb F_q^4, dx)}\\
\lesssim &q^{\frac{2}{3}} \sum_{(k,j)\in I_9}  2^{-k-j} |E_k|^{\frac{5}{6}} |E_j|^{\frac{1}{2}}= q^{\frac{2}{3}} \sum_{(k,j)\in I_9}  2^{-k-j} |E_k||E_j|^{\frac{1}{2}+ \frac{1}{6}} |E_k|^{-\frac{1}{6}}  |E_j|^{-\frac{1}{6}}\\
< &\sum_{(k,j)\in I_9}  2^{-k-j} |E_k||E_j|^{\frac{1}{2}+ \frac{1}{6}} \quad\mbox{since}~~  |E_k|^{-\frac{1}{6}}, |E_j|^{-\frac{1}{6}} < q^{-\frac{1}{3}} ~~\mbox{for}~~(k,j)\in I_9\\
=& \sum_{k=0}^\infty 2^{-k} |E_k| \sum_{j=k}^\infty 2^{-j}|E_j|^{\frac{2}{3}} 
\leq \sum_{k=0}^\infty 2^{-k} |E_k| \sum_{j=k}^\infty 2^{-\frac{j}{5}} \sim \sum_{k=0}^\infty   2^{-\frac{6k}{5}}|E_k|=1, \end{align*}
where we also used (\ref{size}), the convergence of a geometric series, and (\ref{one}) in the last line.
\end{proof}

{\bf Acknowledgment :} The author would like to thank the referee for  his/her valuable comments for developing the final version of this paper.

\end{document}